\date{April 12, 2009}

\documentclass[11pt,reqno]{amsart}
\usepackage{latexsym,amsmath,amsfonts,amscd,amssymb}

\theoremstyle{plain}
\newtheorem{theorem}{Theorem}[section]
\newtheorem{corollary}[theorem]{Corollary}
\newtheorem{lemma}[theorem]{Lemma}
\newtheorem{proposition}[theorem]{Proposition}
\theoremstyle{definition}
\newtheorem{definition}[theorem]{Definition}
\newtheorem{remark}[theorem]{Remark}
\numberwithin{equation}{section}

\newcommand{\s}{\sigma}

\newcommand{\ba}{\mathbf{a}}
\newcommand{\bb}{\mathbf{b}}
\newcommand{\bp}{\mathbf{p}}
\newcommand{\bn}{\mathbf{n}}
\newcommand{\br}{\mathbf{r}}
\newcommand{\bs}{\mathbf{s}}

\newcommand{\bd}{\mathbf{d}}

\newcommand{\bT}{\bar{\bar{T}}}
\newcommand{\bbT}{\bar{\bar{T}}}
\newcommand{\bbr}{\varsigma}

\newcommand{\cD}{\mathcal{D}}

\newcommand{\cM}{\mathcal{M}}

\newcommand{\cN}{\mathcal{N}}

\newcommand{\cO}{\mathcal{O}}

\newcommand{\cS}{\mathcal{S}}

\newcommand{\cU}{\mathcal{U}}

\renewcommand{\AA}{\mathbb{A}}
\newcommand{\QQ}{\mathbb{Q}}

\newcommand{\RR}{\mathbb{R}}
\newcommand{\CC}{\mathbb{C}}
\newcommand{\HH}{\mathbb{H}}

\newcommand{\TT}{\mathbb{T}}
\newcommand{\ZZ}{\mathbb{Z}}

\newcommand{\inc}{\hookrightarrow}
\newcommand{\lto}{\longrightarrow}
\newcommand{\surj}{\twoheadrightarrow}

\newcommand{\x}{\times}
\newcommand{\ox}{\otimes}

\newcommand{\la}{\langle}
\newcommand{\ra}{\rangle}
\newcommand{\frM}{{\frak M}}
\newcommand{\frm}{{\frak m}}
\newcommand{\frh}{{\frak h}}
\newcommand{\frs}{{\frak s}}
\newcommand{\frp}{{\frak p}}

\newcommand{\hs}{{\frh\frs}}
\newcommand{\phs}{{\frp\frh\frs}}
\newcommand{\mhs}{{\frm\frh\frs}}

\newcommand{\Sym}{\mathrm{Sym}}

\newcommand{\Jac}{\mathrm{Jac}}
\DeclareMathOperator{\rk}{rk}\DeclareMathOperator{\rg}{rg}
 \DeclareMathOperator{\im}{im}

\DeclareMathOperator{\Hom}{Hom} \DeclareMathOperator{\Ext}{Ext}
 \DeclareMathOperator{\Gr}{Gr}
\DeclareMathOperator{\GL}{GL}\DeclareMathOperator{\Aut}{Aut}

\newcommand{\scp}{{\s_c^+}}

\newcommand{\scm}{{\s_c^-}}

\title{Hodge structures of the moduli spaces of pairs}

\subjclass[2000]{Primary: 14D20. Secondary: 14H60, 14F45.}
\keywords{Moduli space, complex curve, holomorphic bundle, Hodge
structure.}

\author{Vicente Mu\~noz}
  \address{Instituto de Ciencias Matem{\'a}ticas CSIC-UAM-UC3M-UCM \\
  Consejo Superior de Investigaciones Cient{\'\i}ficas \\ Serrano 113 bis
  \\ 28006 Madrid \\ Spain}
  \address{Facultad de Matem\'{a}ticas \\ Universidad Complutense
  de Madrid \\ Plaza Ciencias 3
  \\ 28040 Madrid \\ Spain}
  \email{vicente.munoz@imaff.cfmac.csic.es}

\thanks{Partially supported through grant MEC
(Spain) MTM2007-63582}

\begin{document}
\maketitle

\begin{abstract}
 Let $X$ be a smooth projective curve of genus $g\geq 2$ over
 $\CC$. Fix $n\geq 2$, $d\in \ZZ$.
 A pair $(E,\phi)$ over $X$ consists of an algebraic vector bundle
 $E$ of rank $n$ and degree $d$ over $X$ and a section $\phi \in H^0(E)$.
 There is a concept of stability for pairs which depends on a
 real parameter $\tau$. Let $\frM_\tau(n,d)$ be the moduli space
 of $\tau$-semistable pairs of rank $n$ and degree $d$ over $X$.
 Here we prove that the cohomology
 groups of $\frM_\tau(n,d)$ 
 are Hodge
 structures isomorphic to direct summands of tensor products of the
 Hodge structure $H^1(X)$. This implies a similar result for the
 moduli spaces of stable vector bundles over $X$.
\end{abstract}

\section{Introduction} \label{sec:introduction}

Let $X$ be a smooth projective curve of genus $g\geq 2$ over the
field of complex numbers. Fix $n \geq 2$ and $d\in\ZZ$. We shall
denote by $M(n,d)$ the moduli space of ($S$-equivalence classes of)
semistable bundles of rank
$n$ and degree $d$ over $X$. The open subset consisting of stable
bundles will be denoted $M^s(n,d) \subset M(n,d)$. Note that
$M(n,d)$ is a projective variety, which is in general not smooth if
$n$ and $d$ are not coprime. On the other hand, $M^s(n,d)$ is a
smooth quasi-projective variety. If $L_0$ is a fixed line bundle of
degree $d$, then we have the moduli spaces $M^s(n,L_0)$ and $M(n,L_0)$
consisting of stable and polystable bundles $E$, respectively, with
determinant $\det(E)\cong L_0$.

A pair $(E,\phi)$ over $X$ consists of a bundle $E$ of rank $n$ and
degree $d$ over $X$ together with a section $\phi\in H^0(E)$. There
is a concept of stability for a pair which depends on the choice of
a parameter $\tau \in \RR$. This gives a collection of moduli spaces
of $\tau$-polystable pairs $\frM_\tau(n,d)$, which are projective
varieties. It contains a smooth open subset $\frM_\tau^s(n,d)\subset
\frM_\tau(n,d)$ consisting of $\tau$-stable pairs. If we fix the
determinant $\det(E)\cong L_0$, then we have the moduli spaces of
pairs with fixed determinant, $\frM_\tau^s(n,L_0)$ and
$\frM_\tau(n,L_0)$. Pairs are discussed at length in
\cite{B,BD,GP,MOV1}.

The range of the parameter $\tau$ is an open interval $I$ split by a
finite number of critical values $\tau_c$. For a non-critical value
$\tau\in I$, there are no properly semistable pairs, so $\frM_\tau
(n,d)=\frM_\tau^s(n,d)$ is smooth and projective. For a critical
value $\tau=\tau_c$, $\frM_\tau(n,d)$ is in general singular at
properly $\tau$-semistable points. Our first main result is the following.

\begin{theorem} \label{thm:main-HS}
  Let $\tau \in I$ be non-critical. Then the pure Hodge structures
  $H^i(\frM_\tau(n,d))$ are isomorphic to direct summands of
  some tensor products of $H^1(X)$.
  A similar result holds for $H^i(\frM_\tau(n,L_0))$.
\end{theorem}

We recover from this a well-known result of Atiyah \cite{At}.

\begin{corollary} \label{cor:main-HS-bundles}
 Suppose that $n$ and $d$ are coprime. Then the
 pure Hodge structures $H^i(M(n,d))$ and $H^i(M(n,L_0))$
 are isomorphic to direct summands of some tensor products
 of $H^1(X)$.
\end{corollary}

Our strategy of proof is the following. First, we find more
convenient to rephrase the problem in terms of triples. A triple
$(E_1,E_2,\phi)$ consists of a pair of bundles $E_1$, $E_2$ of ranks
$n_1,n_2$ and degrees $d_1,d_2$, respectively, and a homomorphism $\phi:E_2\to E_1$.
There is a suitable concept of stability for triples depending on
a real parameter $\sigma$. This gives rise to moduli spaces
$\cN_\sigma=\cN_\sigma (n_1,n_2,d_1,d_2)$ of $\s$-polystable
triples.

There is an identification of moduli spaces of pairs and triples
given by $\frM_\tau (n,d)\to \cN_\s(n,1,d,\cO)$, $(E,\phi)\mapsto
(E,\cO,\phi)$, where $\cO$ is the trivial line bundle, and $\s=
(n+1) \tau - d$. Actually, this rephrasing is a matter of
aesthetic. The arguments are carried out with the moduli spaces of
triples so that they could eventually be generalised to the case
of triples of arbitrary ranks $n_1,n_2$.

The range of the parameter $\s$ is an interval $I =(\s_m,\s_M)
\subset \RR$ split by a finite number of critical values $\s_c$.
When $\s$ moves without crossing a critical value, then
$\cN_{\s}=\cN_\s (n,1,d_1,d_2)$ remains unchanged, but when $\s$,
crosses a critical value, $\cN_{\s}$ undergoes a birational
transformation which we call a \emph{flip}. This consists on
removing some subvariety and inserting a different one. We give a
stratification of the flip locus, and describe explicitly the
strata. This allows us to prove Theorem \ref{thm:main-HS}.

For $\s=\s_m^+=\s_m+\epsilon$, $\epsilon>0$ small, we have a
morphism $\cN_{\sigma} \to M(n,d)$. When $d$ is large enough,
this is a fibration over the locus $M^s(n,d)$. This allows us to
deduce Corollary \ref{cor:main-HS-bundles} from Theorem
\ref{thm:main-HS}.

\medskip

For proving the main result, we introduce the notion of a
mixed Hodge structure to be $R_X$-generated when the graded pieces
are pure Hodge structures which are isomorphic to direct summands of
tensor products of the Hodge structure $H^1(X)$. We actually obtain the following.

\begin{theorem} \label{thm:main-MHS}
  Let $\tau \in I$ (critical or not). Then the mixed Hodge structures
  $H^i(\frM^s_\tau(n,d))$ and $H^i(\frM_\tau^s(n,L_0))$ are $R_X$-generated.

 Let $n\geq 2$, $d\in\ZZ$ (coprime or not). Then the
 mixed Hodge structures $H^i(M^s(n,d))$ and $H^i(M^s(n,L_0))$
 are $R_X$-generated.
\end{theorem}

We prove this by induction on the rank, since the flip loci can be suitable
described in terms of moduli spaces of lower rank. In the course of the proof, we get an
explicit geometrical description of the flip loci. This can be useful for many other
applications:
\begin{itemize}
 \item Extend the results of \cite{Mu2} to compute the Hodge-Deligne polynomials of the moduli
 spaces of pairs for arbitrary rank $n>3$.
 \item Compute the $K$-theory class of the moduli spaces of pairs and the moduli spaces of bundles.
 \item Prove the Generalized Hodge Conjecture for moduli spaces of pairs and bundles
 corresponding to curves $X$ which are generic, by using the result of \cite{AK}.
\end{itemize}

\noindent \textbf{Acknowledgements:}  The author is grateful to the Tata Institute of Fundamental Research
(Mumbai), where part of this work was carried out, for its hospitality.

\section{Hodge structures}
\label{sec:MHS}

Let us start by recalling the Hodge-Deligne theory of algebraic
varieties over $\CC$. See the original reference \cite{De} or the
nice book \cite{St} for generalities on this.

\subsection{Pure Hodge structures} \label{subsec:2.1}

Let $H$ be a finite-dimensional vector space over $\QQ$. A
\emph{pure Hodge structure} of weight $k$ on $H$ is a
decomposition
 $$
 H_\CC=H\ox \CC=\bigoplus\limits_{p+q=k} H^{p,q}\,,
 $$
such that $H^{q,p}=\overline{H}{}^{p,q}$, the bar denoting complex
conjugation in $H$. A Hodge structure of weight $k$ on $H$ gives
rise to the so-called Hodge filtration $F$ on $H_\CC$, where
  $$
  F^p= \bigoplus\limits_{s\geq p} H^{s,k-s}\ ,
  $$
which is a descending filtration. Note that $\Gr_F^p H_\CC =
F^p/F^{p+1}= H^{p,q}$.

Let $\hs$ denote the category of pure Hodge structures. This is an
abelian category with tensor products. We shall denote by 
$\QQ(l)$ the Hodge structure (of weight $2l$) given by the vector
space $H=(2\pi i)^l \QQ$ with $H^{l,l}=\CC$. The trivial Hodge
structure is $\QQ=\QQ(0)$, and the Hodge structure $\TT=\QQ(1)$ is
known as the Tate Hodge structure. For any Hodge structure $H$,
the $l$-th Tate twist is $H(l):=H\ox \TT^{\ox l}=H\ox \QQ(l)$,
which has the same underlying vector space as $H$, but with grading
$H(l)^{p,q}=H^{p-l,q-l}$.

A pure Hodge structure $H$ of weight $k$ is \emph{polarizable} if
there exists a morphism of Hodge structures $\theta: H\otimes H\to
\QQ(k)$, which is a non-degenerate bilinear map. Let $\phs$ denote
the category whose objects are polarised Hodge structures. This is
an abelian sub-category with tensor products of $\hs$. The category $\phs$
is semi-simple, that is, if $H'\subset H$ is a sub-Hodge
structure of a polarised Hodge structure $H$, then $H'$ is also
polarisable, and there exists another sub-Hodge structure
$H''\subset H$ such that $H=H'\oplus H''$.

Consider the Grothendieck group of $\hs$. This is the abelian
group $K_0(\hs)$ generated by elements $[H]$, where $H$ is a pure
Hodge structure, with the relation $[H]=[H']+[H'']$, whenever
there is an exact sequence
 $$
 0\to H'\to H\to H''\to 0.
 $$
Clearly, $K_0(\hs)$ is generated by the \emph{simple} Hodge
structures. A Hodge structure is simple if it does not admit
proper sub-Hodge structures. Note that any polarisable Hodge
structure is a direct sum of simple Hodge structures.

Let $H$ be a pure Hodge structure. Then there is a filtration
$0\subset H_1\subset H_2\subset \cdots \subset H_r=H$ such that
$\bar H_i=H_i/H_{i-1}$ is a simple Hodge structure. In
$K_0(\hs)$ we have that $H$ is equivalent to $\bigoplus\bar H_i$.
Note that the element $\bigoplus \bar H_i$ is uniquely defined (up
to the order of the summands). We define $\Gr(H):=\bigoplus
\bar H_i$.

If $Z$ is a compact smooth projective variety (hence compact
K{\"a}hler) then the cohomology of $Z$, $H^k(Z)$, admits a Hodge
structure given by the Hodge decomposition of harmonic forms into
$(p,q)$ types. This is pure of weight $k$ and it is polarised.

\subsection{Mixed Hodge structures}\label{subsec:2.2}

Let $H$ be a finite-dimensional vector space over $\QQ$. A
\emph{mixed Hodge structure} over $H$ consists of an ascending
weight filtration $W$ on $H$ and a descending Hodge filtration $F$
on $H_\CC$ such that $F$ induces a pure Hodge filtration of weight
$r$ on each rational vector space $\Gr^W_r H= W_r/W_{r-1}$. We
define $H^{p,q} =\Gr_F^p (\Gr^W_{p+q} H)_\CC$.

Let $\mhs$ denote the category of mixed Hodge structures. This is
an abelian category with tensor products. There is a natural map
\cite[(3.1)]{St} from $\mhs$ to $K_0(\hs)$,
 \begin{equation} \label{eqn:star}
  \Psi: H\mapsto \sum_r [\Gr_r^W H]\, ,
  \end{equation}
which sends a mixed Hodge structure to a direct sum of pure Hodge
structures of different weights.

Deligne has shown \cite{De} that, for each complex algebraic
variety $Z$, the cohomology $H^k(Z)$ and the cohomology with
compact support $H_c^k(Z)$ both carry natural (mixed) Hodge structures. If
$Z$ is a smooth projective variety then this is the pure Hodge
structure previously mentioned.

\subsection{Hodge structures $R$-generated}\label{subsec:2.4}

Let $R\subset \hs$ be a given collection of simple Hodge
structures. We define
 $$
 \la R\ra\subset K_0(\hs)
 $$
as the smallest sub-ring containing all elements of $R$, and which is closed by
taking sub-objects (that is, if $\sum n_iH_i\in \la R\ra$, with
$n_i\neq 0$, then $H_i\in \la R\ra$, for all $i$), and which is
closed under Tate twists.

We say that a Hodge structure $H\in \hs$ is $R$-generated if
$[H]\in \la R\ra$. We say that a mixed Hodge structure $H\in\mhs$
is $R$-generated if $\Psi(H)\in \la R\ra$.

In general, we shall take a collection of polarised Hodge
structures $R=\{H_1,\ldots, H_m\}$. In this case, all Hodge
structures in $\la R\ra$ are polarisable. Actually, a simple
polarised Hodge structure $H\in\la R\ra$ is a sub-Hodge structure
of some
  $$
  H\subset H_1^{k_1}\ox \ldots \ox H_m^{k_m} (l),
  $$
for suitable $k_1,\ldots, k_m\geq 0$, $l\in \ZZ$. Equivalently,
there is an epimorphism from a tensor product of elements of $R$
to $H$.

In particular, if $Z$ is a smooth projective variety, and $H^*(Z)$ is $R$-generated,
then $H^k(Z)$ is a sub-Hodge structure of some direct sum of
tensor products of Hodge structures in $R$, for each $k$.

\begin{lemma} \label{lem:hX}
  Let $Z=Y\sqcup W$. If two of the mixed Hodge
  structures $H^*_c(Z)$, $H^*_c(Y)$ and $H^*_c(W)$ are
  $R$-generated, then so is the third.
\end{lemma}

\begin{proof}
We can assume that $Y$ is a closed subvariety of $Z$, and then $W=Z-Y$ is open in $W$.
Suppose that
$H^*_c(Y)$ and $H^*_c(W)$ are $R$-generated, and let us see that
  $H^*_c(Z)$ is also $R$-generated. (The other cases are dealt with in an
  analogous way.) We have a long exact sequence
   $$
   \ldots \to H_c^k(W)  \to H_c^k(Z) \to H_c^k(Y) \to H_c^{k+1}(W)  \to H_c^{k+1}(Z) \to \ldots
   $$
  Since $\Gr^W_r$ is an exact functor \cite{De}, we have an exact sequence for any
  given $r$
   $$
   \ldots \to \Gr_r^W H_c^k(W)  \to  \Gr_r^W H_c^k(Z) \to  \Gr_r^W H_c^k(Y) \to
   \Gr_r^W H_c^{k+1}(W)  \to  \ldots
   $$
  Then we have short exact sequences
   \begin{equation}
   0\to A_k \to \Gr_r^W H_c^k(Z) \to B_k \to 0\, , \label{eqn:s-e-s}
   \end{equation}
  for each $k$. Also we have exact sequences
   $$
   0\to B_k \to \Gr_r^W H_c^k(Y) \to C_k \to 0\, ,
   $$
   $$
   0\to C_k\to \Gr_r^W H_c^{k+1}(W) \to  A_{k+1} \to 0\, .
   $$

 In $K_0(\hs)$ we have that $\Gr(\Gr_W^r
 H_c^k(Y))=\Gr(B_k)\oplus \Gr(C_k)$, hence $\Gr(B_k)$ is also
 $R$-generated, being a sub-Hodge structure of an $R$-generated
 one. The same is true for $\Gr(A_{k+1})$. This happens for all $k$.
 So (\ref{eqn:s-e-s}) implies that $H_c^k(Z)$ is $R$-generated.
\end{proof}

We can apply the above to the set $R_{triv}=\emptyset$. The Hodge structures
in $\la R_{triv} \ra$ are the trivial Hodge structures (that is, those for
which $H^{p,q}=0$ unless $p=q$). Lemma \ref{lem:hX} has the following corollary.

\begin{corollary} \label{cor:trivial-HS}
 Suppose that $Z$ admits a stratification $Z=\sqcup Z_i$ where $Z_i$ is a
 product of affine spaces $\AA^r$ and/or spaces of the form $\AA^r-\AA^s$, $0\leq s< r$.
 Then $H^k_c(Z)$ is $R_{triv}$-generated.
\end{corollary}

\begin{proof}
 This follows from Lemma \ref{lem:hX} and the fact that the Hodge structure of $\AA^r$ is
 trivial.
\end{proof}

\noindent \textbf{Notation:} Let $X$ be a smooth projective complex
curve of genus $g\geq 2$, and consider the Hodge structure $H^1(X)$. This
is a polarised Hodge structure, with the polarisation given by the
cup product. We shall denote $R_X=\{H^1(X)\}$.

\section{Moduli spaces of triples} \label{sec:triples}

Let $X$ be a smooth projective curve of genus $g\geq 2$ over $\CC$.
A triple $T = (E_{1},E_{2},\phi)$ on $X$ consists of two vector
bundles $E_{1}$ and $E_{2}$ over $X$, of ranks $n_1$ and $n_2$ and
degrees $d_1$ and $d_2$, respectively, and a homomorphism $\phi
\colon E_{2} \to E_{1}$. We shall refer to $(n_1,n_2,d_1,d_2)$ as
the {type} of the triple, and $(n_1,n_2)$ as the rank of the triple.

For any $\s \in \RR$, the $\s$-slope of $T$ is defined by
 $$
   \mu_{\s}(T)  =
   \frac{d_1+d_2}{n_1+n_2} + \s \frac{n_{2}}{n_{1}+n_{2}}\ .
 $$
We say that a triple $T = (E_{1},E_{2},\phi)$ is $\s$-stable if
$\mu_{\s}(T') < \mu_{\s}(T)$ for any proper subtriple $T' =
(E_{1}',E_{2}',\phi')$. We define $\s$-semistability by replacing
the above strict inequality with a weak inequality. A triple $T$ is
$\s$-polystable if it is the direct sum of $\s$-stable triples of
the same $\s$-slope. We denote by
  $$
  \cN_\s(n_1,n_2,d_1,d_2)
  $$
the moduli space of $\s$-polystable triples of type
$(n_1,n_2,d_1,d_2)$. This moduli space was constructed in \cite{BGP}
and \cite{Sch}. It is a complex projective variety. The open subset
of $\s$-stable triples will be denoted by
$\cN_\s^s(n_1,n_2,d_1,d_2)$.

Let $L_1,L_2$ be two bundles of degrees $d_1,d_2$ respectively. Then
the moduli spaces of $\sigma$-semistable triples $T=(E_1,E_2,\phi)$
with $\det(E_1)=L_1$ and $\det(E_2)=L_2$ will be denoted
  $$
  \cN_\s(n_1,n_2,L_1,L_2)\, ,
  $$
and $\cN_\s^s(n_1,n_2,L_1,L_2)$ is the open subset of
$\s$-stable triples.

\medskip

Let $\mu(E)=\deg(E)/\rk(E)$ denote the slope of a bundle $E$, and
let $\mu_i=\mu(E_i)=d_i/n_i$, for $i=1,2$. Write
  \begin{align*}
  \s_m = &\, \mu_1-\mu_2\ ,  \\
  \s_M = & \left\{ \begin{array}{ll}
    \left(1+ \frac{n_1+n_2}{|n_1 - n_2|}\right)(\mu_1 - \mu_2)\ ,
      \qquad & \mbox{if $n_1\neq n_2$\ ,} \\ \infty, & \mbox{if $n_1=n_2$\
      ,}
      \end{array} \right.
  \end{align*}
and let $I$ be the interval $I=(\s_m,\s_M)$. Then a necessary
condition for $\cN_\s^s(n_1,n_2,d_1,d_2)$ to be non-empty is that
$\s\in I$ (see \cite{BGPG}). Note that $\s_m>0$. To study the
dependence of the moduli spaces on the parameter $\s$, we need to
introduce the concept of critical value \cite{BGP,MOV1}.

\begin{definition}\label{def:critical}
The values $\s_c\in I$ for which there exist $0 \le n'_1 \leq
n_1$, $0 \le n'_2 \leq n_2$, $d'_1$ and $d'_2$, with $n_1'n_2\neq
n_1n_2'$, such that
 \begin{equation}\label{eqn:sigmac}
 \s_c=\frac{(n_1+n_2)(d_1'+d_2')-(n_1'+n_2')(d_1+d_2)}{n_1'n_2-n_1n_2'},
 \end{equation}
are called \emph{critical values}. We also consider $\s_m$ and $\s_M$
(when $\s_M\neq \infty$) as critical values.
\end{definition}

The interval $I$ is split by a finite number of values $\s_c \in I$.
The stability and semistability criteria  for two values of $\s$
lying between two consecutive critical values are equivalent; thus
the corresponding moduli spaces are isomorphic. When $\s$ crosses a
critical value, the moduli space undergoes a transformation which we
call a \emph{flip}. We shall study the flips in some detail in the
next section.

\begin{theorem}\label{thm:pairs}
 For non-critical values $\s\in I$, $\cN_{\s}=\cN_\s(n,1,d_1,d_2)$ is smooth and
 projective, and it consists only of $\s$-stable points (i.e.
 $\cN_{\s}=\cN_{\s}^s$). For critical values $\s=\s_c$,
 $\cN_{\s}$ is projective, and the open subset
 $\cN_{\s}^s\subset \cN_{\s}$ is smooth. In both cases, the dimension of
 $\cN_{\s}$ is $(n^2-1)(g-1)+d_1-nd_2$.
\end{theorem}


\subsection*{Relationship with pairs}

A pair $(E,\phi)$ over $X$ consists of a vector bundle $E$ of rank
$n$ and 
degree $d$, and $\phi\in H^0(E)$. Let $\tau\in \RR$. We say that
$(E,\phi)$ is $\tau$-stable (see \cite[Definition 4.7]{GP}) if:
 \begin{itemize}
 \item For any subbundle $E'\subset E$, we have $\mu(E')<\tau$.
 \item For any subbundle $E'\subset E$ with $\phi\in H^0(E')$, we
 have $\mu(E/E')>\tau$.
 \end{itemize}
The concept of $\tau$-semistability is defined by replacing the
strict inequalities by weak inequalities. A pair $(E,\phi)$ is
$\tau$-polystable if $E=E'\oplus E''$, where $\phi\in H^0(E')$, $(E',\phi')$
is $\tau$-stable, and
$E''$ is a polystable bundle of slope $\tau$. The moduli space of
$\tau$-polystable pairs is denoted by $\frM_\tau ( n, d)$.

Interpreting $\phi\in H^0(E)$ as a morphism $\phi:\cO \to E$, where
$\cO$ is the trivial line bundle on $X$, we have a map
$(E,\phi)\mapsto (E,\cO,\phi)$ from pairs to triples. The
$\tau$-stability of $(E,\phi)$ corresponds to the $\s$-stability of
$(E,\cO,\phi)$, where $\s=(n+1)\tau -d$ (see \cite{BGP}). Therefore
we have an isomorphism of moduli spaces
   \begin{equation}\label{eqn:isom}
    \cN_\sigma (n,1,d,0)  \cong \frM_\tau (n,d) \x \Jac\, X\, ,
   \end{equation}
given by $(E,L,\phi)\mapsto ((E\ox L^*,\phi),L)$.

\medskip


If $L_0$ is a line bundle of degree $d$, then $\frM_\tau(n,L_0)$ denotes the subspace of
$\frM_\tau(n,d)$ consisting of pairs $(E,\phi)$ where $\det(E) \cong L_0$. Note that
$\cN_\sigma (n,1,L_0,\cO)  \cong \frM_\tau (n,L_0)$.

\section{Flips for the moduli spaces of pairs} \label{sec:flips}

The homological algebra of triples is controlled by the
hypercohomology of a certain complex of sheaves which appears when
studying infinitesimal deformations \cite[Section 3]{BGPG}. Let
$T'=(E'_1,E'_2,\phi')$ and $T''=(E''_1,E''_2,\phi'')$ be two triples
of types $(n_{1}',n_{2}',d_{1}',d_{2}')$ and
$(n_{1}'',n_{2}'',d_{1}'',d_{2}'')$, respectively. Let
$\Hom(T'',T')$ denote the linear space of homomorphisms from $T''$
to $T'$, and let $\Ext^1(T'',T')$  denote the linear space of
equivalence classes of extensions of the form
 $$
  0 \lto T' \lto T \lto T'' \lto 0,
 $$
where by this we mean a commutative  diagram
  $$
  \begin{CD}
  0@>>>E_1'@>>>E_1@>>> E_1''@>>>0\\
  @.@A\phi' AA@A \phi AA@A \phi'' AA\\
  0@>>>E'_2@>>>E_2@>>>E_2''@>>>0.
  \end{CD}
  $$
To analyze $\Ext^1(T'',T')$ one considers the complex of sheaves
 \begin{equation} \label{eqn:extension-complex}
    C^{\bullet}(T'',T') \colon ({E_{1}''}^{*} \otimes E_{1}') \oplus
  ({E_{2}''}^{*} \otimes E_{2}')
  \overset{c}{\lto}
  {E_{2}''}^{*} \otimes E_{1}',
 \end{equation}
where the map $c$ is defined by
 $$
 c(\psi_{1},\psi_{2}) = \phi'\psi_{2} - \psi_{1}\phi''.
 $$

We introduce the following notation:
\begin{align*}
  \HH^i(T'',T') &= \HH^i(C^{\bullet}(T'',T')). 
\end{align*}
%

\begin{proposition}[{\cite[Proposition 3.1]{BGPG}}]
  \label{prop:hyper-equals-hom}
  There are natural isomorphisms
  \begin{align*}
    \Hom(T'',T') &\cong \HH^{0}(T'',T'), \\
    \Ext^{1}(T'',T') &\cong \HH^{1}(T'',T'),
  \end{align*}
and a long exact sequence associated to the complex
$C^{\bullet}(T'',T')$:
 $$
 \begin{array}{c@{\,}c@{\,}c@{\,}l@{\,}c@{\,}c@{\,}c}
  0 &\lto \mathbb{H}^0(T'',T') &
  \lto & H^0(({E_{1}''}^{*} \otimes E_{1}') \oplus ({E_{2}''}^{*}
\otimes
  E_{2}'))
  & \lto &  H^0({E_{2}''}^{*} \otimes E_{1}') \\
    &  \lto \mathbb{H}^1(T'',T') &
  \lto &  H^1(({E_{1}''}^{*} \otimes E_{1}') \oplus
({E_{2}''}^{*} \otimes
  E_{2}'))
 &  \lto & H^1({E_{2}''}^{*} \otimes E_{1}') \\
 &   \lto \mathbb{H}^2(T'',T') & \lto & 0. & &
 \end{array}
 $$ 
\end{proposition}

We shall use the following results later:

\begin{lemma}[] \label{lem:H0=0}
  Suppose that $T'$, $T''$ are triples such that $T'$ is $\s$-stable
  and $T''$ is $\s$-semistable. Then $\HH^0(T',T'')=0$ unless
  $T''$ contains a subtriple isomorphic to $T'$, in which case
  $\HH^0(T',T'')=\CC$.
\end{lemma}

\begin{lemma}[{\cite[Lemma 3.10]{Mu}}] \label{lem:H2=0}
 If $T''=(E_1'',E_2'',\phi'')$ is an injective triple, that is $\phi'':E_2''\to E_1''$ is
 injective, then $\HH^2(T'',T')=0$.
\end{lemma}

Fix the type $(n_1,n_2,d_1,d_2)$ for the moduli spaces of triples.
For brevity, write $\cN_\s=\cN_\s(n_1,n_2,d_1,d_2)$. Let $\s_c\in
I$ be a critical value and set
 $$
 \scp = \s_c + \epsilon,\quad \scm = \s_c -
 \epsilon,
 $$
where $\epsilon > 0$ is small enough so that $\s_c$ is the only
critical value in the interval $(\scm,\scp)$.

\begin{definition}\label{def:flip-loci}
We define the \textit{flip loci} as
 \begin{align*}
 \cS_{\scp} &= \{ T\in\cN_{\scp}^s \ ;
 \ \text{$T$ is $\scm$-unstable}\} \subset\cN_{\scp}^s \ ,\\
 \cS_{\scm} &= \{ T\in\cN_{\scm}^s \ ;
 \ \text{$T$ is $\scp$-unstable}\}
 \subset\cN_{\scm}^s \, .
 \end{align*}
\end{definition}
It follows that (see \cite[Lemma 5.3]{BGPG})
 $$
 \cN_{\scp}^s-\cS_{\scp}=\cN_{\s_c}^s=\cN_{\scm}^s-\cS_{\scm}.
 $$

For a triple $T$, we denote $\lambda(T)=\frac{n_2}{n_1+n_2}$.
It $T\in \cS_{\s_c^+}$, then there is a subtriple $T'\subset T$
satisfying
 $$
 \mu_{\s_c}(T')=\mu_{\s_c}(T) , \text{and} \ \lambda(T')< \lambda(T).
  $$
Conversely, if $T\in \cS_{\s_c^-}$, then there exists a subtriple
$T'\subset T$ satisfying
 $$
 \mu_{\s_c}(T')=\mu_{\s_c}(T) , \text{and} \ \lambda(T') > \lambda(T).
  $$
Clearly, in both cases, $T$ is properly $\s_c$-semistable.

\bigskip

From now on we shall restrict to the case of triples of rank $(n,1)$.
Denote
 $$
 \cN_\s=\cN_\s(n,1,d,d_o),
 $$
the moduli space of $\s$-semistable triples $(E,L,\phi)$, where $L$ is a line bundle
of degree $d_o$, $E$ is a bundle of rank $n$ and degree $d$, and $\phi:L\to E$.
Let $\s_c$
be a critical value for the moduli spaces of type $(n,1,d,d_o)$.
We want
to describe a stratification of $\cS_{\s_c^\pm}$ by suitable
locally closed subvarieties.

\begin{lemma} \label{lem:cS}
Let $\s_c$ be a critical value.
  Let $T\in \cS_{\s_c^\pm}$. Then there is a unique filtration
  $0=T_0\subset T_1\subset T_2\subset \cdots \subset T_r=T$ such that $T_i/T_{i-1}$ is
  the maximal $\s_c$-polystable subtriple of $T/T_{i-1}$ of the same
  $\s_c$-slope.
\end{lemma}

\begin{proof}
  We only need to prove that there is a unique maximal $\s_c$-polystable
  subtriple for any given triple $T$.
  If $T_1$ and $T_1'$ are two $\s_c$-polystable subtriples of $T$,
  then the subtriple
  $T_1+T_1'\subset T$ sits in an exact sequence
   \begin{equation}\label{eqn:a}
   0\to T_1\cap T_1' \to T_1\oplus T_1' \to T_1+T_1' \to 0\,.
   \end{equation}
  As $\mu_{\s_c} (T_1\oplus T_1') \leq \mu_{\s_c}(T_1+T_1') \leq
  \mu_{\s_c}(T )=\mu_{\s_c}(T_1)=\mu_{\s_c}(T_1')$, we have that all
  the triples in (\ref{eqn:a}) have the same $\s_c$-slope. Therefore
  $T_1+T_1'$ is a quotient of the $\s_c$-polystable triple $T_1\oplus T_1'$,
  hence it is itself $\s_c$-polystable. This contradicts the
  maximality of either $T_1$ or $T_1'$.
\end{proof}

\begin{definition}
We shall call \emph{standard filtration} to the filtration
provided by Lemma \ref{lem:cS}. We shall write
$\bar{T}_i=T_i/T_{i-1}$, $i=1,\ldots, r$.
\end{definition}

\begin{remark}
The result in Lemma \ref{lem:cS} is true for triples of any type
$(n_1,n_2,d_1,d_2)$, for any $\s$ and any
$T \in \cN_\s$. Note that if $T$ is $\s$-stable then the
standard filtration is $0\subset T$. 
\end{remark}

\begin{lemma} \label{lem:4.5}
   Let $0=T_0\subset T_1\subset T_2\subset \cdots \subset T_r=T$ be the
   standard filtration of $T$, and let $\bar{T}_i=T_i/T_{i-1}$, $i=1,\ldots, r$.
   \begin{itemize}
   \item If $T\in \cS_{\s_c^+}$ then $\bar T_r$ is a $\s_c$-stable
   triple of type $(n',1)$, and $\bar T_i$, $1\leq i<r$, are
   triples of the form $(F_i,0,0)$, where $F_i$ are polystable
   bundles all of the same slope. Moreover, $n'>0$ if $\s_c>\s_m$ and $n'=0$ if $\s_c=\s_m$.
   \item If $T\in \cS_{\s_c^-}$ then $\bar T_1$ is a $\s_c$-stable
   triple of type $(n',1)$, and $\bar T_i$, $1<i\leq r$, are
   triples of the form $(F_i,0,0)$, where $F_i$ are polystable
   bundles all of the same slope. Moreover, always $n'>0$.
   \end{itemize}
\end{lemma}

\begin{proof}
Let $T=(E,L,\phi)$, and recall that $L$ is a line bundle of
degree $d_o$.

Let $T\in \cS_{\s_c^+}$, and let $0=T_0\subset T_1\subset
T_2\subset \cdots \subset T_r=T$ be its standard filtration. Then
$T/T_{r-1}$ is a quotient of $T$. As $T$ is $\s_c^+$-stable, $\bar
T_r=T/T_{r-1}$ is of rank $(n',1)$ and $\mu_{\s_c}(\bar
T_r)=\mu_{\s_c}(T)=:\mu_c$. Therefore, $T_{r-1}$ is of rank
$(n-n',0)$, and hence all $\bar T_i$ are of rank $(n_i,0)$, for
$i<r$. So all of the triples $\bar T_i$ are direct sums of triples
of the form $(F_{ij},0,0)$, where $F_{ij}$ are stable bundles of
the same slope, $1\leq i\leq r-1$, $1\leq j\leq k_i$. Note that
 $$
 \mu(F_{ij})=\mu_{\s_c} (\bar T_i) = \mu_c\,,
 $$
so all the bundles are of the same slope. Moreover, $\bar T_r$
should be $\s_c$-polystable. It cannot have summands of rank
$(n'',0)$ and $\s_c$-slope $\mu_c$, since this would imply that
$T$ has such a quotient, and hence it cannot be $\s_c^+$-stable.
So $\bar T_r$ is $\s_c$-stable triple of the form $(E',L,\phi')$.

If $\s_c>\s_m$ then $\phi'$ is injective, so $n'>0$. If $\s_c=\s_m$, then
\cite[Proposition 4.10]{MOV1} tells us that all $\s_c$-semistable triples
are not $\s_c$-stable, unless $n'=0$.

Let $T\in \cS_{\s_c^-}$, and let $0=T_0\subset T_1\subset
T_2\subset \cdots \subset T_r=T$ be its standard filtration. Then
$T_1\subset T$ must be of rank $(n',1)$, by the $\s_c^-$-stability
of $T$. In principle, we could say that $T_1$ is $\s_c$-polystable, but it
is actually $\s_c$-stable, since otherwise
it has a subtriple of type $(n'',0)$ but then $T$ cannot be $\s_c^-$-stable.
Therefore $T/T_1$ is of rank $(n-n',0)$, and hence all
$\bar T_i$ are of rank $(n_i,0)$, for $i>1$. So all of the triples
$\bar T_i$ are direct sums of triples of the form $(F_{ij},0,0)$,
where $F_{ij}$ are stable bundles of the same slope
$\mu(F_{ij})=\mu_{\s_c}(\bar{T}_i)=\mu_c:=\mu_{\s_c}(T)$. Moreover
$\mu_{\s_c}(\bar T_1)=\mu_c$ and $T_1$ is a $\s_c$-stable triple
of the form $(E', L,\phi')$. Note that $\s_c>\s_m$ (since we are
dealing with $\cS_{\s_c^-}$), so $n'>0$.
\end{proof}

\section{Stratification of the flip loci}\label{sec:strata}

Let $\s_c$ be a critical value. In this section we want to
describe a stratification of the flip loci $\cS_{\s_c^\pm}$.
Recall that we are dealing with the moduli spaces $\cN_\s=\cN_\s(n,1,d,d_o)$.

\subsection{The flip locus $\cS_{\s_c^+}$}

Now we want to describe geometrically $\cS_{\s_c^+}$. We stratify
$\cS_{\s_c^+}$ according to the types of the triples in the
standard filtration. Let us fix some notation: let $b\geq 1$ and
$r\geq 1$. Fix $n'\geq 1$ and $d'$ such that
 \begin{equation}\label{eqn:1}
 \frac{d'+\s_c}{n'+1} = \mu_{\s_c}(T) = : \mu_c \,  .
 \end{equation}
Let $(n_1,d_1),\ldots, (n_b,d_b)$ satisfy
 \begin{equation}\label{eqn:2}
 \frac{d_i}{n_i}=\mu_c\,.
 \end{equation}
Consider $a_{ij}\geq 0$, for $1\leq i\leq b$ and $1\leq j\leq r$,
such that $(a_{1j},\ldots,a_{bj})\neq (0,\ldots,0)$, for all $j$.
We assume that
 \begin{equation}\label{eqn:3}
 \sum_{i,j} a_{ij} n_i + n'= n.
 \end{equation}

Write $\ba_j=(a_{1j},\ldots , a_{bj})$, $1\leq j\leq r$, and
$\bn=\{(n_i,d_i,a_{ij}) \ ; \ 1\leq i\leq b, 1\leq j \leq r\}$.
Consider
 \begin{equation}\label{eqn:3.5}
 \tilde U(\bn)= \{(E_1,\ldots, E_b)\in M^s(n_1,d_1)\times \cdots \times
 M^s(n_b,d_b) \, ; \, E_i\not\cong E_j, \ \text{for } i\neq j \}\,.
 \end{equation}
For each $(E_1,\ldots, E_b)\in \tilde U(\bn)$, set
$S_i=(E_i,0,0)$, $1\leq i\leq b$.

We define $X^+(\bn)\subset \cS_{\s_c^+}$ as the subset formed by
those triples $T$ whose standard filtration is $0=T_0\subset
T_1\subset T_2\subset \cdots \subset T_{r+1}=T$, such that
  \begin{equation} \label{eqn:b}
  \bar{T}_j=T_j/T_{j-1} \cong S(\ba_j):= S_1^{a_{1j}} \oplus \cdots  \oplus
  S_b^{a_{bj}}\, ,
  \end{equation}
for some $(E_1,\ldots, E_b)\in \tilde U(\bn)$. Note that it must
be $\bar T_{r+1}\in \cN_{\s_c}^s(n',1,d',d_o)$. By Lemma
\ref{lem:4.5}, the subsets $X^+(\bn)$ stratify $\cS_{\s_c^+}$, for
all possibilities for $\bn$ as above. That is,
  $$
  \cS_{\s_c^+}=\bigsqcup_{\bn} X^+(\bn)\, .
  $$

To describe $X^+(\bn)$, first note that the presentation
(\ref{eqn:b}) is not unique. The finite group
   \begin{equation} \label{eqn:F}
   F=\{\tau \text{ permutation of } (1,\ldots, b) \, ; \,
   n_{\tau(i)}=n_i, a_{\tau(i)j}=a_{ij} \ \forall i,j \ \}
  \end{equation}
acts (freely) on $\tilde U(\bn)$ by permuting the bundles. Let
${U}(\bn)=\tilde U(\bn)/F$. There is a map
  \begin{equation} \label{eqn:c}
 X^+(\bn)\to U(\bn) \x \cN_{\s_c}^s(n',1,d',d_o).
  \end{equation}
The pull-back of (\ref{eqn:c}) under $\tilde U(\bn)\to U(\bn)$ is
 \begin{equation*} 
 \tilde{X}^+(\bn)\to \cM(\bn):= \tilde U(\bn) \x \cN_{\s_c}^s(n',1,d',d_o).
  \end{equation*}
and $F$ acts on $\tilde X^+(\bn)$ so that
$X^+(\bn)=\tilde{X}^+(\bn)/F$.

We shall construct an affine bundle over $\cM(\bn)$
parametrizing iterated extensions, and an open subset
$\tilde{\tilde{X}}^+(\bn)$ of it, parametrizing those extensions
corresponding to $\s_c^+$-stable triples. In this way, we shall
get quotient maps
 $$
 \tilde{\tilde{X}}^+(\bn)\to \tilde{X}^+(\bn) \to X^+(\bn).
 $$
The first quotient corresponds to the automorphisms of the
extensions, and the second one to permuting the order of the
$S_i$'s.

To do this, we start by describing the elements in
$X^+(\bn)$ more explicitly.

\begin{proposition} \label{prop:Xbn}
  Let $(E_1,\ldots, E_b, \bar T_{r+1})\in \cM(\bn)$.
  Define triples $\bT_j$ by downward recursion as follows:
  $\bT_{r+1}=\bar T_{r+1}$ and for $1\leq j\leq r$
  define $\bT_j$ as an extension
   \begin{equation}\label{eqn:ext}
   0\to S(\ba_j)\to \bT_j \to \bT_{j+1}\to 0\, .
   \end{equation}
  Let $\xi_j\in \Ext^1(\bT_{j+1},S(\ba_j))$ be the extension class corresponding to (\ref{eqn:ext}).
  Denote $T=\bT_1$. Then $T\in X^+(\bn)$ if and only if the
  following conditions are satisfied:
  \begin{enumerate}
  \item The extension class $\xi_j\in \Ext^1(\bT_{j+1},S(\ba_j))= \prod_i
  \Ext^1(\bT_{j+1},S_i)^{a_{ij}}$ lives in
  $\prod_i V(a_{ij},\Ext^1(\bT_{j+1},S_i))$, with the notation \newline
  $V(k,W)=\{(w_1,\ldots, w_k)\in W^k \ ; \ w_1,\ldots, w_k \text{
  are linearly independent}\}$, \newline for $W$ a vector space.
  \item Consider the
   map $S(\ba_{j+1})\to \bT_{j+1}$ and the element $\xi_j'$ which
   is the image of $\xi_j$ under
   $\Ext^1(\bT_{j+1},S(\ba_j)) \to \Ext^1(S(\ba_{j+1}),S(\ba_j))$.
   Then the class $\xi_j'\in \Ext^1(S(\ba_{j+1}),S(\ba_j))= \prod_i
  \Ext^1(S_i,S(\ba_{j}))^{a_{i,j+1}}$ lives in $\prod_i
  V(a_{i,j+1},\Ext^1(S_i,S(\ba_{j})))$.
  \end{enumerate}

  Two extensions $\xi_j$ give rise to isomorphic $\bT_j$
  if and only if the triples $\bT_{j+1}$ are
  isomorphic and the extension classes are the same up to action
  of the group $\GL(\ba_j):=\GL(a_{1j})\times \cdots \times \GL(a_{bj})$.
\end{proposition}

\begin{proof}
Assume that $T\in X^+(\bn)$. So $T$ is $\s_c^+$-stable. As $\bT_j$
is a quotient of $T$ with the same $\s_c$-slope, it is also
$\s_c^+$-stable (if not, there would be a quotient $\bbT_j\surj
T''$ with the same $\s_c$-slope, but with $\lambda(T'')=0$; this
would violate the $\s_c^+$-stability of $T$ since $T\surj T''$).

Now let us prove (1) by downward induction on $j$. If $\xi_j$ does
not live in $\prod_i V(a_{ij},\Ext^1(\bT_{j+1},S_i))$, then there
is some $i_o$ and some linear projection $\prod_i
\Ext^1(\bT_{j+1},S_i)^{a_{ij}} \to \Ext^1(\bT_{j+1},S_{i_o})$
under which $\xi_j$ maps to zero. This corresponds to a quotient
$S(\ba_j)\to S_{i_o}$. Let $\tilde S$ be the subbundle defined as
the kernel $\tilde S \inc S(\ba_j)\to S_{i_o}$. Then the
class $\xi_j$ lies in the image of $\Ext^1(\bT_{j+1},\tilde S)\inc
\Ext^1(\bT_{j+1},S(\ba_j))$. This produces an extension $0\to
\tilde S\to \tilde{T}\to \bbT_{j+1}\to 0$, such that
$\tilde{T}\subset \bbT_j$ is a subtriple
with $\lambda(\tilde{T})>0$, violating
the $\s_c^+$-stability of $\bbT_j$.

Conversely, $T$ defined as in the statement is $\s_c$-semistable.
All $\bbT_j$ are extensions of $\s_c$-semistable objects of the same
$\s_c$-slope, so all of them are $\s_c$-semistable of the same
$\s_c$-slope. Assume now that we know that $\bbT_{j+1}$ is
$\s_c^+$-stable. If $\bbT_j$ is not $\s_c^+$-stable, then there
should be a subtriple $\tilde{T}\subset \bbT_j$ with
$\lambda(\tilde{T})>0$ and $\mu_{\s_c}(\tilde
T)=\mu_{\s_c}(\bbT_j)$. This lies in a exact sequence
 $$
 \begin{array}{ccccccccc}
   0 &\to &\tilde{S}&\to& \tilde{T}&\to& \tilde{T}''&\to& 0\\
     & & \cap & &\cap & &\cap \\
   0 &\to& S(\ba_j)&\to& \bT_j& \to &\bT_{j+1}&\to &0\, .
 \end{array}
 $$
It is clear that all the triples are of the same $\s_c$-slope, and
that $\lambda(\tilde{T}'')>0$. Therefore the $\s_c^+$-stability of
$\bT_{j+1}$ implies that $\tilde{T}''=\bT_{j+1}$. Therefore the
extension $\xi_j$ maps to zero under
  $$
  \Ext^1(\bT_{j+1}, S(\ba_j)) \to \Ext^1(\bT_{j+1},S(\ba_j)/\tilde S)\, .
  $$
Take $i_o$ such that there is a quotient $S(\ba_j)/\tilde S \surj S_{i_o}$. We
see that $\xi_j$ maps to zero under $\Ext^1(\bT_{j+1}, S(\ba_j))
\to \Ext^1(\bT_{j+1},S_{i_o})$, so $\xi_j\not\in\prod_i V(a_{ij},
\Ext^1(\bT_{j+1},S_i))$.

\bigskip

Regarding (2), assume that $T\in X^+(\bn)$ at let us see that the class
$\xi_j'$ satisfies the condition (2) by downward induction on $j$.
First suppose that $\xi_j'\not\in \prod_i
V(a_{i,j+1},\Ext^1(S_i,S(\ba_{j})))$. Then there is some
$S_{i_o}\subset S(\ba_{j+1})$ such that $\xi_j$ maps to zero under
$\prod_i \Ext^1(S(\ba_{j+1}),S(\ba_{j})) \to
\Ext^1(S_{i_o},S(\ba_{j}))$. The pull-back of the extension
(\ref{eqn:ext}) under $S_{i_o}\inc S(\ba_{j+1})\subset
\bbT_{j+1}$,
 $$
 \begin{array}{ccccccccc}
   0 &\to & S(\ba_j)&\to& \tilde{T}&\to& S_{i_o} &\to& 0\, { }\\
     & & || & &\cap & &\cap \\
   0 &\to& S(\ba_j)&\to& \bT_j& \to &\bT_{j+1}&\to &0,
 \end{array}
 $$
is split. Therefore there is a polystable subtriple
$S(\ba_j)\oplus S_{i_o} \subset \bT_j$. This violates the fact that the
standard filtration of $T$ is of given type $\bn$.

Conversely, assume that (2) is satisfied and let us see that the
standard filtration of $T$ is of type $\bn$. We can see this by downward
induction on $j$. Assume that it is known for $j+1$.
We consider the exact sequence $S(\ba_j)\to
\bbT_j\to \bbT_{j+1}$. If the maximal $\s_c$-polystable subtriple
of $\bbT_j$ is not $S(\ba_j)$, then there is some $i_o$ such that
$S(\ba_j)\oplus S_{i_o}\subset \bbT_j$. Therefore there is some
inclusion $S_{i_o}\inc \bbT_{j+1}$ such that the extension class
$\xi_j$ maps to zero under
 $$
 \Ext^1(\bT_{j+1} , S(\ba_j) )\to \Ext^1(S_{i_o}, S(\ba_j))\, .
 $$
Clearly, $S_{i_o}\subset S(\ba_{j+1})\subset \bbT_{j+1}$.
Therefore $\xi_j'$ does not satisfy (2).

\bigskip

Finally, let us prove the last assertion. Again the uniqueness of
the standard filtration implies that if two triples $T,T'$ are
isomorphic, then the associated triples $\bbT_j$, $\bbT_j'$ are
isomorphic for all $j$. By downward induction, we can assume that
$\bbT_{j+1}=\bbT_{j+1}'$. We have isomorphisms
  \begin{equation}\label{eqn:4}
 \begin{array}{ccccccccc}
   0 &\to & S(\ba_j) &\to & \bbT_{j} &\to & \bbT_{j+1} &\to & 0\\
      && \cong\downarrow\,\,   && \cong\downarrow\,\,  && || \,\, \\
   0 &\to & S(\ba_j) &\to & \bbT_{j}' &\to & \bbT_{j+1}' &\to & 0
  \end{array}
  \end{equation}
As $\bT_j$ is $\s_c^+$-stable, the only automorphisms of it are
multiplication by non-zero scalars. On the other hand,
$\Aut(S(\ba_j))=\GL(a_{1j}) \times \cdots \times
\GL(a_{bj})=\GL(\ba_j)$. This group acts on the space of
extensions. After an action by an element of $\GL(\ba_j)$, we can
assume that the first vertical arrow of (\ref{eqn:4}) is an
equality. Therefore the extensions are equivalent.
\end{proof}

\begin{remark} \label{rem:sm+}
 The description in Proposition \ref{prop:Xbn} can be applied to the
 critical value $\s_c=\s_m$. In this case,
 $\cN_{\s_m^+}=\cS_{\s_m^+}$. The only difference is that now
 $\bT_{r+1}$ should be of the form $L\to 0$, that is,
 $\bT_{r+1}\in \cN_{\s_c}^s(0,1,0,d_o)=\Jac^{d_o}X$.

 Note that there is an open stratum in $\cS_{\s_m^+}$ corresponding
 to $r=1$, $b=1$, $\ba_1=(a_{11}=1)$, $\bn_o=(n,d,a_{11})$. In this case
  $$
  \cM(\bn_o)=M^s(n,d)\times \Jac^{d_o} X.
  $$
 It corresponds to triples
 $\phi:L\to E$ for which $E$ is a stable bundle. We shall denote this
 open stratum and its complement as
  $$
  \cU_m\subset \cN_{\s_m^+}\, , \qquad \cD_m=\cN_{\s_m^+}-\cU_m\,.
  $$
%
%
\end{remark}

\subsection{The flip locus $\cS_{\s_c^-}$}

There is an analogous description of $\cS_{\s_c^-}$. As before,
let $b\geq 1$ and $r\geq 1$. Fix $n'\geq 1$ and $d'$ satisfying
(\ref{eqn:1}). Let $(n_1,d_1),\ldots, (n_b,d_b)$ satisfy
(\ref{eqn:2}). Consider $a_{ij}\geq 0$, for $1\leq i\leq b$ and
$2\leq j\leq r+1$, such that $\ba_j=(a_{1j},\ldots,a_{bj})\neq
(0,\ldots,0)$, for all $j$. We assume (\ref{eqn:3}). Write
$\bn=\{(n_i,d_i,a_{ij}) \ ; \ 1\leq i\leq b, 2\leq j \leq r+1\}$.
We consider $\tilde U(\bn)$ as in (\ref{eqn:3.5}).

We define $X^-(\bn)\subset \cS_{\s_c^-}$ as the subset formed by
those triples $T$ whose standard filtration is $0=T_0\subset
T_1\subset T_2\subset \cdots \subset T_{r+1}=T$, such that
  \begin{equation} \label{eqn:b2}
  \bar{T}_j=T_j/T_{j-1} \cong S(\ba_j):= S_1^{a_{1j}} \oplus \cdots  \oplus
  S_b^{a_{bj}}\, ,
  \end{equation}
for some $(E_1,\ldots, E_b)\in U(\bn)$, $2\leq j\leq r+1$. It must
be $T_{1}\in \cN_{\s_c}^s(n',1,d',d_o)$. By Lemma \ref{lem:4.5},
the subsets $X^-(\bn)$ stratify $\cS_{\s_c^-}$, for all possible
choices of $\bn$. That is,
  $$
  \cS_{\s_c^-}=\bigsqcup_{\bn} X^-(\bn)\, .
  $$

Note again that the finite group $F$ given in (\ref{eqn:F}) acts
on $\tilde U(\bn)$, and that the pull-back of the map
   $$
 X^-(\bn)\to \tilde U(\bn) \x \cN_{\s_c}^s(n',1,d',d_o).
  $$
under $\tilde U(\bn) \to U(\bn)=\tilde U(\bn)/F$ is
 \begin{equation*}
 \tilde{X}^-(\bn)\to \cM(\bn):= \tilde U(\bn) \x \cN_{\s_c}^s(n',1,d',d_o).
  \end{equation*}
and $X^-(\bn)=\tilde{X}^-(\bn)/F$.

The proof of the following result is analogous to that of Proposition \ref{prop:Xbn}.

\begin{proposition} \label{prop:Xbn-}
  Let $(E_1,\ldots, E_b, T_{1})\in \cM(\bn)$.
  Define triples $\bT_j$ by recursion as follows:
  $\bT_{1}=T_{1}$, and for $2\leq j\leq r+1$
  define $\bT_j$  as an extension
   $$
   0\to \bT_{j-1} \to \bT_{j}\to S(\ba_j)\to 0\, .
   $$
  Let $\xi_j\in \Ext^1(S(\ba_j),\bT_{j-1})$ be the corresponding extension class.
  Denote $T=\bT_{r+1}$. Then $T\in X^-(\bn)$ if and only if the
  following conditions are satisfied:
  \begin{enumerate}
  \item The extension class $\xi_j\in \Ext^1(S(\ba_j),\bT_{j-1})= \prod_i
  \Ext^1(S_i,\bT_{j-1})^{a_{ij}}$ lives in
  $\prod_i  V(a_{ij}, \Ext^1(S_i,\bT_{j-1}))$.
   \item Consider the
   map $\bT_{j-1}\to S(\ba_{j-1})$ and the element $\xi_j'$ which
   is the image of $\xi_j$ under
   $\Ext^1(S(\ba_j),\bT_{j-1}) \to \Ext^1(S(\ba_{j}),S(\ba_{j-1}))$.
   Then the element $\xi_j'\in \Ext^1(S(\ba_{j}),S(\ba_{j-1}))= \prod_i
  \Ext^1(S(\ba_{j}),S_i)^{a_{i,j-1}}$ lives in $\prod_i
  V(a_{i,j-1}, \Ext^1(S(\ba_{j}),S_i))$.
  \end{enumerate}

  Two extensions $\xi_j$ give rise to isomorphic $\bT_{j}$
  if and only if the triples $\bT_{j-1}$ are
  isomorphic and the extension classes are the same up to action
  of the group $\GL(\ba_j):=\GL(a_{1j})\times \cdots \times \GL(a_{bj})$.
\end{proposition}

\begin{remark} \label{rem:sM-}
 This description is also valid for $\s_c=\s_M$, with no change.
 In this case, $\cN_{\s_M^-}=\cS_{\s_M^-}$.

 Again for $\s_c=\s_M$, there is an open stratum, corresponding to the case $r=1$, $b=1$,
 $(n',d')=(1,d_o)$, $\ba_1=(a_{11}=1)$, $\bn_o=(n-1,d-d_o,1)$. Then
 $\cN_{\s_c}^s(1,1,d_o,d_o)=\Jac^{d_o}(X)$, and
 $(n_1,d_1)=(n-1,d-d_o)$, so
  $$
  \cM(\bn_0)= M^s(n-1,d-d_o)\x \Jac^{d_o} X\, .
  $$
 The elements of $\tilde X^-(\bn_0)=X^-(\bn_0)$ correspond to extensions
  $$
  \begin{array}{ccccccccc}
  0 &\to & L & \to & L & \to & 0 &\to  &0 \\
  && \downarrow && \downarrow && \downarrow \\
  0 &\to & L & \to & E & \to & F &\to  &0,
  \end{array}
  $$
  where $F\in M^s(n-1,d-d_o)$ and $L\in \Jac^{d_o} X$. This is the
  open stratum considered in \cite[section 7.2]{BGPG}, which is a
  dense open subset of $\cN_{\s_M^-}$.
\end{remark}

\section{$R_X$-generation of the flip loci}\label{sec:class-in-K}

Let $F$ be a quasi-projective variety. We say that $F$ is \emph{affinely stratified}
(abbreviated A.S.) if there is a stratification $F=\sqcup F_i$ such that
each $F_i$ is a product of varieties of the form $\AA^r$ or $\AA^r-\AA^s$, $0\leq s<r$.
Note that Corollary \ref{cor:trivial-HS} implies that its cohomology is
a trivial Hodge structure.

\medskip

Let $\s_c$ be a critical value, and fix a type $\bn$ corresponding
to it. We want to construct an affine bundle over $\tilde U(\bn)$
parametrizing iterated extensions, and an open subset
$\tilde{\tilde{X}}^+(\bn)$ of it, parametrizing those extensions
corresponding to $\s_c^+$-stable triples. We start with
 $$
 M_{r+1}=\cM(\bn) =\tilde U(\bn) \x \cN_{\s_c}^s(n',1,d',d_o)\, ,
 $$
and define spaces $M_j$ by downward induction as follows: $M_j$ is
the space parametrizing
extensions (\ref{eqn:ext}). More precisely, if $M_{j+1}$ is already
constructed, then
we define a fiber bundle
  $$
  \pi_j:\bar M_{j}\to M_{j+1}
  $$
whose fibers are the spaces
$\Ext^1(\bbT_{j+1},S(\ba_j))$. Then any element in $\bar M_{j}$ gives
rise to a triple $\bbT_j$. Let $M_{j}\subset\bar M_{j}$ be
the subset of those extensions $\xi_j$ satisfying (1) and (2) of
Proposition \ref{prop:Xbn}, that is, when the corresponding triple
$\bbT_j$ is $\s_c^+$-stable. We denote
  $$
  p_j:M_j\to M_{j+1}
  $$
the restriction of $\pi_j$ to $M_j$.
The space we are interested in is $\tilde{\tilde{X}}^+(\bn):=M_1$.

\begin{theorem}\label{thm:1}
  The fibration $p_j:M_{j}\to M_{j+1}$ is a locally
  trivial fibration (in the Zariski topology) whose fiber $F_j$ is A.S.
\end{theorem}

\begin{proof}
  Fix a point in $M_{j+1}$. This determines a $\s_c^+$-stable
  triple $\bbT_{j+1}$. Note that
  $\HH^0(\bbT_{j+1},S(\ba_j))=0$ since $\bbT_{j+1}$ is $\s_c^+$-stable,
  $\bbT_{j+1}$ and $S(\ba_j)$ are of the same $\s_c$-slope and
  $S(\ba_j)$ is $\s_c$-polystable and does not
  contain a copy of $\bbT_{j+1}$ (see Lemma \ref{lem:H0=0}). Also $\HH^2(\bbT_{j+1},S(\ba_j))=0$
  by Lemma \ref{lem:H2=0}. So $\dim \Ext^1(\bbT_{j+1},S(\ba_j))$ is constant, and
  therefore
  $\pi_j:\bar M_j\to M_{j+1}$ is a vector bundle, locally
  trivial in the Zariski topology.

  We want to understand the fiber $F_j$ of $p_j:M_{j}\to M_{j+1}$,
  which consists of those extensions in $\Ext^1(\bbT_{j+1},S(\ba_j))$
  satisfying (1) and (2) of Proposition \ref{prop:Xbn}. The
  triple $\bbT_{j+1}$, being an
  element in $M_{j+1}$, actually sits in an exact sequence
   $$
   0\to \bbT_{j+2}\to \bbT_{j+1}\to S(\ba_{j+1}) \to 0.
   $$
  So there is an associated long exact sequence
   \begin{equation}\label{eqn:5}
   \begin{aligned}
   0 \to \Hom(S(\ba_{j+1}),S(\ba_j)) & \to
   \Ext^1(\bbT_{j+2},S(\ba_j)) \to \\
   &\to \Ext^1(\bbT_{j+1}, S(\ba_j)) \to
   \Ext^1(S(\ba_{j+1}),S(\ba_j))\to 0\, .
   \end{aligned}
   \end{equation}
Clearly, $\dim\Hom(S(\ba_{j+1}),S(\ba_j))=\sum_i a_{ij}a_{i,j+1}$
is constant (independent of the point in the base space
$M_{j+1}$), and so the dimensions of all the spaces in
(\ref{eqn:5}) are constant.

We aim to prove that $F_j$ is A.S. Note that if we have a fibration
$F\to E\to B$ in which both $F$ and $B$ are A.S., then so is $E$
(stratify $B$ by affine sets, and note that the fibration should be
trivial over each one of them).

We obtain that $F_j$ is A.S. as a consequence of the following claims.

\bigskip \noindent \textbf{Claim 1.} \emph{Let $a,b\geq 1$ be integers. Consider
   $$
   M_{ab}(\CC)=\{f: \CC^a\ox \CC^b\to \CC\}.
   $$
   Let $0\leq r \leq \min\{a,b\}$ and $U_r=\{f\in M_{ab}(\CC) \ ;
   \ \rg (f)=r\}$. Then $U_r]$ is A.S. There are vector bundles $K_1,K_2 \to U_r$
   (of ranks $a-r$, $b-r$ resp.),
   defined by
   \begin{eqnarray*}
    K_{1,f} &=& \ker (f:\CC^a \to (\CC^b)^*), \\
    K_{2,f} &=& \ker (f:\CC^b \to (\CC^a)^*), \quad \text{for $f\in U_r\subset M_{ab}(\CC)$.}
    \end{eqnarray*}
   }

 \begin{proof}  This is well-known. Let us give a short proof for completeness.
   Take a full flag $0\subset F_1\subset
   F_2\subset \cdots \subset F_a=\CC^a$. Fix integers $0\leq
   d_1\leq \ldots \leq d_a=r$, with the
   condition that $d_{i+1}-d_i\leq 1$, and write $\bd=(d_1,\ldots, d_a)$.
   Consider the set
    $$
    U_\bd= \{ f \ ; \ \rg(f|_{F_j})=d_j, \forall j \}\, .
    $$
   First we see that $U_\bd$ is A.S. Define $U_{\bd,l}=\{ f:F_l\to (\CC^b)^*
   \ ; \ \rg(f|_{F_j})=d_j, \forall j\leq l \}$, $1\leq l\leq a$. Then $U_{\bd,l+1}\to
   U_{\bd,l}$ is a fiber bundle, whose fiber is
    $$
    \left\{ \begin{array}{l}
    \Hom(F_{l+1}/F_l,\im(f|_{F_l})), \quad \text{if
    $d_{l+1}=d_l$},\\[5pt]
    \Hom(F_{l+1}/F_l,\CC^b)-\Hom(F_{l+1}/F_l,\im(f|_{F_l})), \quad \text{if $d_{l+1}=d_l+1$}.
    \end{array} \right.
    $$
   The first case is an affine space, the second is the difference
   of two affine spaces. So both are A.S. spaces. This
   implies that $U_{\bd}=U_{\bd,a}$ and
   $U_r=\bigsqcup_{\bd} U_{\bd}$ are A.S. The assertion
   about $K_1,K_2$ is clear.
 \end{proof}

  \bigskip
  \noindent \textbf{Claim 2.} \emph{Let $a,b\geq 1$ be integers and let
  $V$ be
  a finite dimensional vector space. Consider
   $$
   M_{ab}(V)=\{f: \CC^a\ox \CC^b\to V\}\, .
   $$
  Let $0\leq r\leq a$, $0\leq s\leq b$, and define the subset
  $$
  U_{r,s}=\{f\in M_{ab}(V)\ ; \ \rg (f:\CC^a \to (\CC^b)^*\ox V)=r,
  \rg (f:\CC^b \to (\CC^a)^*\ox V)=s\}.
  $$
  Then $U_{r,s}$ is A.S. There are vector bundles $K_1,K_2 \to U_{r,s}$
  (of ranks $a-r$, $b-s$ resp.), defined by
   \begin{eqnarray*}
    K_{1,f}&=&\ker (f:\CC^a \to (\CC^b)^*\ox V),\\
   K_{2,f}&=&\ker (f:\CC^b \to (\CC^a)^*\ox V), \quad \text{for $f\in U_{r,s}$.}
   \end{eqnarray*}
  }

 \begin{proof}
 Consider a decomposition $V=V_1\oplus \ldots \oplus V_k$, into
$1$-dimensional vector subspaces $V_i$. Let $\bar V_l=V_1\oplus
\ldots \oplus V_l$, $1\leq l\leq k$, and denote by $\pi_l: V\to
\bar V_l$ the projection. For collections of integers $0\leq
r_1\leq \ldots\leq r_k=r$ and $0\leq s_1\leq \ldots\leq s_k=s$,
denote $\br=(r_1,\ldots, r_k)$, $\bs=(s_1,\ldots,s_k)$, and define
the subsets:
 $$
 \begin{aligned}
  U_{\br,\bs,l}=\{f: \CC^a\ox \CC^b\to \bar V_l\ ; \ & \rg
 (\pi_i\circ f:\CC^a \to (\CC^b)^*\ox \bar V_i)=r_i, \\
  & \rg (\pi_i\circ f:\CC^b \to (\CC^a)^*\ox \bar V_i)=s_i,
  \forall i\leq l\}\, ,
  \end{aligned}
  $$
for $1\leq l\leq k$, and $U_{\br,\bs}:=U_{\br,\bs,k}$. Let us
check that $U_{\br,\bs,l}$ is A.S. For $l=1$, this is the
content of Claim 1 (in case $r_1\neq s_1$, this set is empty). For
$l>0$, we have a natural map
 $$
 U_{\br,\bs,l+1}\to U_{\br,\bs, l}, \quad f\mapsto \pi_l\circ f\,
 .
 $$
By induction, there are vector bundles $K_1,K_2\to U_{\br,\bs,
l}$, of ranks $a-r_l$, $b-s_l$, resp., such that $K_{1,f_0}=\ker
(f_0:\CC^a \to (\CC^b)^*\ox \bar V_l)$, $K_{2,f_0}=\ker (f_0:\CC^b
\to (\CC^a)^*\ox \bar V_l)$, for $f_0\in U_{\br,\bs,l}$.

Consider now $0\leq i\leq \min\{r_{l+1}-r_l,s_{l+1}-s_l\}$, and
define
 $$
 U_{\br,\bs,l+1}^i = \{ f \in U_{\br,\bs,l+1}  \ ; \
 \rg(f: {K_{1,\pi_l\circ f}\ox K_{2,\pi_l\circ f}\to \bar V_{l+1}/\bar V_l=
 V_{l+1}})=i \} \, .
 $$
Clearly, we have an stratification $U_{\br,\bs,l+1}=\bigsqcup_i
U_{\br,\bs,l+1}^i$. The set
 $$
 B_i= \{ (f_0,f_1)\ ; \ f_0\in U_{\br,\bs,l},
 f_1: {K_{1,f_0}\ox K_{2,f_0}\to V_{l+1}}, \rg(f_1)=i \}
 $$
is a bundle over $U_{\br,\bs,l}$, and the fibers are A.S. by
Claim 1. So $B_i$ are A.S. There are vector bundles $K_1,K_2\to
B_i$ defined by
 $$
 \begin{aligned}
  & K_{1,(f_0,f_1)}= \ker (f_1: K_{1,f_0}\to (K_{2,f_0})^*\ox
  V_{l+1})\subset K_{1,f_0}, \\
  & K_{2,(f_0,f_1)}= \ker (f_1: K_{2,f_0}\to
  (K_{1,f_0})^*\ox V_{l+1})\subset K_{2,f_0} ,
  \end{aligned}
  $$
and they are of ranks $a-r_l-i$, $b-s_l-i$, resp.

There is a natural map $U_{\br,\bs,l+1}^i\to B_i$, which is a
fiber bundle. To see this, consider $(f_0,f_1) \in B_i$, and let
us find the fiber over it. Fix decompositions
 \begin{equation}\label{eqn:6}
  \begin{aligned}
 & \CC^a =(\CC^a/K_{1,f_0}) \oplus (K_{1,f_0}/K_{1,(f_0,f_1)}) \oplus K_{1,(f_0,f_1)}, \\
 & \CC^b= (\CC^b/K_{2,f_0}) \oplus (K_{2,f_0}/K_{2,(f_0,f_1)}) \oplus
K_{2,(f_0,f_1)} \, .
    \end{aligned}
 \end{equation}
Then the fiber over $(f_0,f_1)$ consists of maps $g:\CC^a \ox
\CC^b \to V_{l+1}$ with nine components $g_{ij}$, according to the
decompositions (\ref{eqn:6}). The components $g_{11}, g_{12},
g_{21}$ can be chosen arbitrarily (thus they move in a vector
space). The component
 $$
 g_{22} : (K_{1,f_0}/K_{1,(f_0,f_1)}) \ox (K_{2,f_0}/K_{2,(f_0,f_1)}) \to
 V_{l+1}
 $$
is fixed and equal to $f_1$. The components $g_{23},g_{32},
g_{33}$ are zero (since they should equal $f_1$).  Finally, the
components
 $$
 \begin{aligned}
  & g_{13}: K_{1,(f_0,f_1)} \to (\CC^b/K_{2,f_0})^* \ox V_{l+1},\\
 & g_{31}: K_{2,(f_0,f_1)} \to (\CC^a/K_{1,f_0})^* \ox V_{l+1},
  \end{aligned}
  $$
must satisfy $\rg(g_{13})=r_{l+1}-r_l-i$ and
$\rg(g_{31})=s_{l+1}-s_l-i$. This implies that they move in spaces
which are A.S., by Claim 1.
 \end{proof}

  \bigskip
  \noindent \textbf{Claim 3.} \emph{Let $a_1,\ldots, a_k,b \geq 1$ be
  integers and let $V_1,\ldots, V_k$
  be finite dimensional vector spaces. Put $\ba=(a_1,\ldots,a_k)$, and consider
   $$
   M_{\ba b}= \bigoplus_{i=1}^k \{f_i:\CC^{a_i}\ox \CC^b\to V_i\}\,
   .
   $$
  Let $0\leq r_i\leq a_i$, and write $\br=(r_1,\ldots, r_k)$.
  Define
   $$
   \begin{aligned}
   U_{\br}= \{f=(f_1,\ldots , f_k) \in M_{\ba b} \ ; \ &
   \rg (f_i:\CC^{a_i} \to (\CC^b)^*\ox V_i)=r_i, \forall i, \\
  & \rg (f:\CC^b \to \oplus_i \ (\CC^{a_i})^*\ox V_i)=b \}.
   \end{aligned}
   $$
  Then $U_{\br}$ is A.S., and $K_i \to U_{\br}$, defined by
  $K_{i,f}=\ker (f:\CC^{a_i} \to (\CC^b)^*\ox V_i)$, for $f\in
  U_{\br}$, are vector bundles of ranks $a_i-r_i$, resp.}

 \begin{proof}
 The case $k=1$ is covered by Claim 2. So suppose $k>1$. Consider
 integers $0\leq p_1\leq \ldots \leq p_k=b$, and write
 $\bp=(p_1,\ldots, p_k)$. For $1\leq l\leq k$, define
   $$
   \begin{aligned}
    U_{\br,\bp,l}=
    \{f=(f_1,\ldots , f_l)\in & \, \bigoplus_{i=1}^l \Hom(\CC^{a_i}\ox \CC^b,V_i) \ ; \\
    & \rg (f_i:\CC^{a_i} \to (\CC^b)^*\ox V_i)=r_i, 1\leq i \leq l, \\
    & \rg (f:\CC^b \to \oplus_{i=1}^l (\CC^{a_i})^*\ox V_i)=p_l
    \}.
   \end{aligned}
   $$
  and $U_{\br,\bp}=U_{\br,\bp,k}$. We have a stratification
  $U_{\br}=\bigsqcup_{\bp} U_{\br,\bp}$. Note that there is a
  vector bundle $C_l\to U_{\br,\bp,l}$, defined by
  $C_{l,f_0}=\ker (f_0:\CC^b \to \bigoplus_{i=1}^l (\CC^{a_i})^*\ox
  V_i)$, for $f_0\in U_{\br,\bp,l}$, of rank $b-p_l$.

There is a fibration
 $$
 U_{\br,\bp,l+1}\to U_{\br,\bp,l}.
 $$
The fiber over $f_0\in U_{\br,\bp,l}$ consists of those
$f_{l+1}\in \Hom(\CC^b\ox \CC^{a_{l+1}} , V_{l+1})$ such that
$f_{l+1}: \CC^{a_{l+1}} \to (\CC^b)^*\ox V_{l+1}$ has rank
$r_{l+1}$ and $f_{l+1}|_{C_{l,f_0}} :C_{l,f_0} \subset \CC^b\to
(\CC^{a_{l+1}})^* \ox V_{l+1}$ has rank $p_{l+1}-p_l$.


Working similarly as in Claim 2, we can see that this fiber is A.S. \end{proof}

  \bigskip
  \noindent \textbf{Claim 4.} \emph{Take $k,q\geq 1$.
  Let $a_1,\ldots, a_k\geq 1$ and $b_1,\ldots, b_q\geq 1$ be
  integers, and let $V_{it}$
  be finite dimensional vector spaces, $1\leq i\leq k$, $1\leq t\leq q$.
  Write $\ba=(a_1,\ldots,a_k)$, $\bb=(b_1,\ldots, b_q)$ and consider
   $$
   M_{\ba \bb}= \bigoplus_{i,t} \{f_{it}:\CC^{a_i}\ox \CC^{b_t}\to
   V_{it}\}.
   $$
  Let $0\leq r_i\leq a_i$, and put $\br=(r_1,\ldots, r_k)$. Define
   $$
   \begin{aligned}
   U_{\br}= \{f=(f_{it}) \in M_{\ba \bb} \ ; \ & \rg ((f_{i_0t})_t:\CC^{a_{i_0}} \to
   \oplus_t
   (\CC^{b_t})^*\ox V_{i_0t})=r_{i_0}, \forall i_0 , \\
   &\rg ((f_{it_0})_i:\CC^{b_{t_0}} \to \oplus_i (\CC^{a_i})^*\ox V_{it_0})=b_{t_0},
   \forall t_0 \}.
   \end{aligned}
   $$
  Then $U_{\br}$ is A.S., and $K_i \to U_{\br}$, defined by
  $K_{i,f}=\ker ((f_{it})_t:\CC^{a_{i}} \to \bigoplus_t
   (\CC^{b_t})^*\ox V_{it})$, are vector bundles of ranks $a_i-r_i$. }

\begin{proof} Clearly the case  $q=1$ corresponds to Claim 3. For each
$i$, consider integers $0\leq r_{i1}\leq \ldots \leq r_{iq}=r_i$,
and abbreviate $\bbr=(r_{it})$. Define the sets
   $$
   \begin{aligned}
   U_{\bbr,l}= \{f & \in \bigoplus_{i, 1\leq t\leq l}  \Hom(\CC^{a_i}\ox \CC^{b_t}, V_{it})
   \ ; \  \rg (f:\CC^{a_i} \to \oplus_{t=1}^s (\CC^{b_t})^*\ox
   V_{it}) = r_{is}, \\ & 1\leq i \leq k, 1\leq s\leq l, \
  \rg (f:\CC^{b_t} \to \oplus_i (\CC^{a_i})^*\ox V_{it})=b_t, 1\leq t\leq l
  \},
     \end{aligned}
  $$
and $U_{\bbr}=U_{\bbr,q}$. Clearly $U_{\br}=\bigsqcup_{\bbr}
U_{\bbr}$.

We only need to see that $U_{\bbr,l}$ is A.S. The map
$U_{\bbr,l+1}\to U_{\bbr,l}$ is a fiber bundle. Let $f_0\in
U_{\bbr,l}$. Then $C_{i,f_0}=\ker (f_0: \CC^{a_i} \to
\bigoplus_{t=1}^l (\CC^{b_t})^*\ox V_{it})$ define vector bundles of
rank $a_i-r_{i\, l}$, for all $i$.

The fiber of $U_{\bbr,l+1}\to U_{\bbr,l}$ over  $f_0\in
U_{\bbr,l}$ consists of those $\tilde f=(f_{i,l+1})_i \in
\bigoplus_i \Hom ( \CC^{a_i}\ox \CC^{b_{l+1}} , V_{i,l+1})$ such
that
 $$
 \tilde f:\CC^{b_{l+1}} \to \oplus_i (\CC^{a_i})^*\ox V_{i,l+1}
 $$
has rank $b_{l+1}$, and
 $$
 f_{i,l+1}|_{C_{i,f_0}} : C_{i,f_0}\subset \CC^{a_i} \to
 (\CC^{b_{l+1}})^*\ox V_{i,l+1}
 $$
has rank $r_{i,l+1}-r_{i\, l}$, for all $i$. This gives an A.S. space,
as can be seen with an argument similar to the previous cases.
 \end{proof}

  \bigskip
  \noindent \textbf{Claim 5.} \emph{Take $k,q\geq 1$, $\ba=(a_1,\ldots,a_k)$, $\bb=(b_1,\ldots, b_q)$
  as before. Let $V_{it}$ and $W_i$ be finite dimensional vector spaces, $1\leq i\leq k$,
  $1\leq t\leq q$. Consider
   $$
   \widetilde M_{\ba \bb}= \bigoplus_{i,t} \{f_{it}:\CC^{a_i}\ox \CC^{b_t}\to
   V_{it}\} \oplus \bigoplus_i \{g_i:\CC^{a_i}\to W_i\},
   $$
  and
 $$
 \begin{aligned}
   \widetilde U= \{f=(f_{it},g_i) \in & \, \widetilde M_{\ba \bb} \ ; \ \rg
   \left((f_{i_0t},g_{i_0})_t:\CC^{a_{i_0}} \to
   \left( \oplus_t  (\CC^{b_t})^*\ox V_{i_0t} \right) \oplus W_{i_0}\right)
  \\ & =a_{i_0}, \forall i_0 , \
   \rg \left((f_{it_0})_i:\CC^{b_{t_0}} \to \oplus_i (\CC^{a_i})^*\ox V_{it_0}\right)
   =b_{t_0}, \forall t_0 \}.
 \end{aligned}
    $$
  Then $\widetilde U$ is A.S.}

\begin{proof}
  There is a map $\pi:\widetilde M_{\ba \bb}\to M_{\ba\bb}$, and
  denote $\widetilde{U}_{\br}=\pi^{-1}(U_{\br})\cap
  \widetilde{U}$, for each $\br$, with notations as in Claim 4.
  Then $\widetilde{U}=\bigsqcup_{\br}
  \widetilde{U}_{\br}$, and the maps $\widetilde{U}_{\br} \to
  U_{\br}$ are fibrations whose fibers are easily seen to be A.S.
  The fiber over $f_0\in U_{\br}$ is given by $(g_i) \in
  \bigoplus_i \Hom(\CC^{a_i},W_i)$ such that
  $g_i|_{K_{i,f_0}}:K_{i,f_0}\to W_i$ is injective for all $i$.
\end{proof}

Now we apply these results to the fiber $F_j$ of our map
 $$
 p_j:M_j\to M_{j+1}.
 $$
For a point in $M_{j+1}$, the exact sequence (\ref{eqn:5}) splits
as a direct sum of exact sequences
   \begin{equation*}
   \begin{aligned}
   0 \to \Hom(S(\ba_{j+1}),S_i) & \to
   \Ext^1(\bbT_{j+2},S_i) \to \\
   &\to \Ext^1(\bbT_{j+1}, S_i) \to
   \Ext^1(S(\ba_{j+1}),S_i)\to 0\, .
   \end{aligned}
   \end{equation*}
according to $S(\ba_j)=\oplus S_i^{a_{ij}}$. Therefore, we can set
 \begin{eqnarray*}
 V_{it} &=& \Ext^1(S_t, S_i)^* , \\
 W_i&=&(\Ext^1(\bbT_{j+2},S_i)/\Hom(S(\ba_{j+1}),S_i))^* , \\
 \ba&=&\ba_{j}, \\
 \bb&=&\ba_{j+1}.
 \end{eqnarray*}
Then, with the notations of Claim 5, we have that
$\Ext^1(\bbT_{j+1},S(\ba_j))=\widetilde M_{\ba\bb}$, and the
extensions in $\Ext^1(\bbT_{j+1},S(\ba_j))$ giving rise to triples
$\bbT_j$ which are $\s_c^+$-stable are
 $$
  F_j=\Ext^1(\bbT_{j+1},S(\ba_j)) \cap M_{j+1} =\widetilde{U}\subset
 \widetilde M_{\ba\bb},
 $$
by Proposition \ref{prop:Xbn}. So the fiber of the map $p_j$ is
A.S., as required.
\end{proof}

\begin{remark} \label{rem:sc=sm}
  The same description works for $\s_c=\s_m$ with the condition
  $d/n-d_2>2g-2$. Let us see this.

 By Remark \ref{rem:sm+}, the elements in $M_{r+1}$ are of the
 form $\bbT_{r+1}=(0,L,0)$, so $M_{r+1}=\Jac^{d_o} X$. When
 studying the map $\pi_r: \bar M_r\to M_{r+1}$, we cannot now
 apply Lemma \ref{lem:H2=0} to prove the vanishing of $\HH^2$.
 However, the condition  $d/n-d_2>2g-2$ ensures that, for any
 $S(\ba_r)=(F_r,0,0)$, where $F_r$ is a polystable bundle of slope $d/n$,
 we have
  $\HH^2(\bbT_{r+1},S(\ba_r))=H^1(F_r \ox L^*)=0$, using Proposition
  \ref{prop:hyper-equals-hom}. 
  With this information, we have that $\pi_r$ is a fiber bundle,
  and the rest of the argument can be carried out as before.
\end{remark}

\bigskip

The construction of $\tilde{\tilde{X}}^-(\bn)$ is entirely
similar. Start with
 $$
 M_{1}=\cM(\bn)=\tilde U(\bn)\x \cN_{\s_c}^s(n',1,d',d_o)\, .
 $$
Define by induction $M_j$ as the space parametrizing the triples
$\bbT_j$ of Proposition \ref{prop:Xbn-}, as follows: consider the
vector bundle $\pi_j:\bar M_{j}\to M_{j-1}$ whose fibers are the
spaces $\Ext^1(S(\ba_j),\bbT_{j-1})$. Then let $M_{j}\subset\bar
M_{j}$ be the subset corresponding to those extensions $\xi_j$
satisfying (1) and (2) of Proposition \ref{prop:Xbn-}. The space we are
interested in is $\tilde{\tilde{X}}^-(\bn):=M_{r+1}$.

We have the following result.

\begin{theorem}\label{thm:1-}
  Consider the fibration $p_j:M_{j}\to M_{j-1}$. Then $p_j$ is a locally
  trivial fibration (in the Zariski topology), and the fiber $F_j$ is an A.S. space
  \hfill  $\Box$
\end{theorem}

\section{Hodge structures of moduli of triples} \label{sec:HS-moduli}

Let $X$ be a complex curve of genus $g\geq 2$. We aim to find
information on the Hodge structures of the cohomology of the moduli spaces
$\cN_\s(n,1,d,d_o)$, for non-critical values $\s$, and on the
moduli spaces $M(n,d)$.
Consider the Hodge structure $H^1(X)$ and set
 $$
 R_X=\{H^1(X)\}\ .
 $$
Let us see that the Hodge structures of the moduli spaces of pairs
and of bundles are $R_X$-generated.

\begin{lemma} \label{lem:W-gen-n=1}
The cohomology of $\Jac\, X$ is $R_X$-generated. The cohomology of
$\Sym^k X$ is $R_X$-generated, for any $k\geq 1$.
\end{lemma}

\begin{proof}
The first assertion is clear, since $H^*(\Jac\, X)$ is the
exterior algebra on $H^1(X)$. For the case of the symmetric
products of $X$, note that Macdonald \cite{McD} proves that there
is an epimorphism of rings
  \begin{equation}\label{eqn:mor}
  \bigwedge\nolimits^* H^1(X) \otimes \QQ[\theta] \surj H^*(\Sym^k X)\,,
  \end{equation}
where $\theta$ is the hyperplane class, which is defined as the
class of the divisor $D=\Sym^{k-1}X\subset \Sym^kX$. Clearly, this
is a class in $H^{1,1}(\Sym^kX)\cap H^2(\Sym^k X, \ZZ)$, hence
(\ref{eqn:mor}) is a morphism of Hodge structures. This proves
that $H^*(\Sym^kX)$ is $R_X$-generated.
%
%
%
\end{proof}

Note that $M(1,d)=\Jac\, X$ and $\cN_{\s}^s(1,1,d_1,d_2)=\Jac\,
X\x \Sym^{d_2-d_1} X$. So Lemma \ref{lem:W-gen-n=1} serves as
starting point for an induction (on the rank) to prove that all
Hodge structures of all moduli spaces are $R_X$-generated.

\begin{proposition}\label{prop:W-gen}
 Let $n\geq 2$. Assume that all the Hodge structures
$H^*_c(M^s(n',d'))$ and $H^*_c(\cN_{\s}^s(n',1,d',d_o))$ are
 $R_X$-generated for $n'<n$. Then for any critical value
 $\s_c>\s_m$ and for any type $\bn$, the Hodge structures
$H^*_c(X^+(\bn))$ are
 $R_X$-generated. The same happens for $\s_c=\s_m$ and for any
 type $\bn\neq \bn_0=(n,d,1)$ ($\bn_0$ corresponds to the stratum
 $\cU_m\subset \cN_{\s_m^+}$, c.f.\ Remark \ref{rem:sm+}).

 The same result holds for the Hodge structures $H^*_c(X^-(\bn))$.
\end{proposition}

\begin{proof}
 If all $H^*_c(M^s(n',d'))$ and $H^*_c(\cN_{\s}^s(n',1,d',d_o))$ are
 $R_X$-generated for $n'<n$, then all $H^*_c(\tilde U(\bn))$ and $H^*_c(\cM(\bn))$
are also $R_X$-generated. 
 By Theorem \ref{thm:1},
 $\tilde{\tilde X}^+(\bn)\to \cM(\bn)$ is a fibration whose
 fiber $F$ is A.S. By Corollary \ref{cor:trivial-HS}, $H^*_c(F)$ is
 a direct sum of trivial Hodge structures. Using the Leray spectral
 sequence, we get that $H^*_c(\tilde{\tilde{X}}^+(\bn))$ is $R_X$-generated.

Now we want to see that $H^*_c({\tilde{X}}^+(\bn))$ is $R_X$-generated. First
note that by Poincar\'e duality, this is equivalent to see that
$H^*({\tilde{X}}^+(\bn))$ is $R_X$-generated.
Now $\tilde{\tilde X}^+(\bn)\to \tilde{X}^+(\bn)$ is a locally
trivial fibration, with fibers $\GL=\prod \GL(\ba_j)$. Therefore
there is a fibration
 $$
 \tilde{\tilde X}^+(\bn)\to \tilde{X}^+(\bn) \to P\GL\, ,
 $$
where $P\GL$ is the classifying space for $\GL$.  By the Leray
spectral sequence, the cohomology of $\tilde{X}^+(\bn)$ is generated by that of
$\tilde{\tilde X}^+(\bn)$ and the pull-back of the cohomology
in $H^*(P\GL)$. Let us see that the pull-back of $H^*(P\GL)$ is a trivial
Hodge structure. Consider the following spaces $\hat{\hat{M}}_j$.
First $\hat{\hat{M}}_{r+1}
=M_{r+1}$, and $\hat{\hat{M}}_{j}\to \hat{\hat{M}}_{j+1}$ are fibrations with fibers
$\hat{\hat{F}}_j=\prod_i V(a_{ij},\Ext^1(\bbT_{j+1}, S_i))$. That is, extensions
as in Proposition \ref{prop:Xbn} but only satisfying property (1).
Clearly, $M_1\subset \hat{\hat{M}}_1$. On the other hand, we still have an
action of $\GL$ on $\hat{\hat{M}}_1$, whose quotient ${\hat{M}}_1$ is constructed as
follows: ${\hat{M}}_{r+1}
=M_{r+1}$, and ${\hat{M}}_{j}\to {\hat{M}}_{j+1}$ are fibrations with fibers
$\hat{F}_j=\prod_i \Gr(a_{ij},\Ext^1(\bbT_{j+1}, S_i))$.

There is a diagram whose rows are fibrations:
 $$
 \begin{array}{cccccc}
  \tilde{\tilde X}^+(\bn) &\to & \tilde{X}^+(\bn) & \to &P\GL \, \,  \\
  \cap & & \cap  & & || \\
  \hat{\hat{M}}_1 & \to & {\hat{M}}_1 & \to & P\GL\, .
 \end{array}
 $$
So the pull-back of the cohomology of $P\GL$ to $\tilde{X}^+(\bn)$ factors
through the cohomology of $\hat{M}_1$, which is an iterated Grassmannian
bundle. By Corollary \ref{cor:trivial-HS}, and the usual A.S. of the
Grassmannian given by Schubert cells, we have that the cohomology of the
fiber is a trivial Hodge structure. As the cohomology of the base of this
fibration is $R_X$-generated, we get that
the cohomology $H^*_c(\tilde{X}^+(\bn))$ is $R_X$-generated.

 Finally, $X^+(\bn)=\tilde{X}^+(\bn)/F$, with $F$ defined in (\ref{eqn:F}). Thus we have that
 $H^*_c(X^+(\bn))=H^*_c(\tilde{X}^+(\bn))^F$ is $R_X$-generated as well.
\end{proof}

\begin{proposition} \label{prop:W-gen-2.2}
 If $H^*_c(X^\pm(\bn))$ are $R_X$-generated, for all critical
 values $\s_c$ and for all types $\bn$ (with the exception of $\s_c=\s_m$ and
 $\bn\neq \bn_o=(n,d,1)$), then
 $H^*_c(M^s(n,d))$ and $H^*_c(\cN_{\s}^s(n,1,d,d_o))$, for any $\s$, are
 $R_X$-generated.
\end{proposition}

\begin{proof}
This is a repeated application of Lemma \ref{lem:cS}. For
$\s_c>\s_m$, $\cS_{\s_c^\pm}=\bigsqcup X^\pm(\bn)$, so the (mixed) Hodge
structures of these spaces are $R_X$-generated. The same argument
works for $\cD_m=\bigsqcup_{\bn\neq \bn_0} X^+(\bn)\subset
\cN_{\s_m^+}$.

This implies first that $\cN_{\s_M^-}=\cS_{\s_M^-}$ has cohomology
$H_c^*(\cN_{\s_M^-})$ which is $R_X$-generated.
Then we work inductively on $\s$. If $H_c^*(\cN_{\s})$ is $R_X$-generated
for any $\s>\s_c$, then $H_c^*(\cN_{\s_c}^s)$ is $R_X$-generated, as
$\cN_{\s_c}^s=\cN_{\s_c^+}-\cS_{\s_c^+}$, and then
$H_c^*(\cN_{\s_c^-})$ is
$R_X$-generated as $\cN_{\s_c^-}=\cN_{\s_c}^s\sqcup
\cS_{\s_c^-}$.

Finally $H_c^*(\cN_{\s_m^+})$ is $R_X$-generated. But $H_c^*(\cD_m)$ is also
$R_X$-generated. So $H_c^*(\cU_m)$ is $R_X$-generated, since
$\cU_m=\cN_{\s_m^+}-\cD_m$. Taking $d/n-d_o>2g-2$,
we have a projective fibration $\cU_m\to M^s(n,d)\x \Jac^{d_o}X$, by
\cite[Proposition 4.10]{MOV1} (see Remarks \ref{rem:sm+} and \ref{rem:sc=sm}).
So $H_c^*(M^s(n,d))$ is $R_X$-generated as well.
\end{proof}

Propositions \ref{prop:W-gen} and \ref{prop:W-gen-2.2} together
prove Theorem \ref{thm:main-MHS}, and this implies in particular
Theorem \ref{thm:main-HS} and Corollary \ref{cor:main-HS-bundles}.

\begin{remark} \label{rem:fixed-det}
  The same kind of question can be asked about the moduli spaces
  of triples and of bundles of fixed determinant. Let $M(n,L_0)$
  denote the moduli space of semistable bundles $E$ over $X$ of rank
  $n$ and satisfying $\det(E) \cong L_0$, where $L_0$ is a fixed line bundle of degree $d$.
  Consider also the moduli space $\cN_\s(n,1,L_1,L_2)$ of
  $\s$-polystable triples $(E_1,L_2,\phi)$ with $E_1$ a rank $n$
  bundle with $\det(E_1)=L_1$, where $L_1,L_2$ are fixed line
  bundles of degrees $d_1,d_2$, resp.

  The same techniques developed in Sections \ref{sec:triples}--\ref{sec:HS-moduli} can be
  used to prove that the Hodge structures of $M^s(n,L_0)$ and of
  $\cN_\s^s(n,1,L_1,L_2)$ are $R_X$-generated.
\end{remark}

\end{document}